\documentclass[12pt,a4paper]{article}
\usepackage{amsmath,amssymb,amsthm,hyperref}
\usepackage{graphicx}
\usepackage{version} 
\usepackage{color}
\usepackage{pstricks-add}
\usepackage[bf]{caption}     
\usepackage{geometry}        
\usepackage{bm}  
\usepackage[utf8]{inputenc}
  \newtheorem{theorem}{Theorem}[section]
 \newtheorem{corollary}[theorem]{Corollary}
 \newtheorem{claim}[theorem]{Claim}
 \newtheorem{lemma}[theorem]{Lemma}
 \newtheorem{proposition}[theorem]{Proposition}
 \newtheorem{example}[theorem]{Example}
 {\theoremstyle{definition}
 \newtheorem{definition}[theorem]{Definition}
 \newtheorem{remark}[theorem]{Remark}

 }

\newcommand{\beq}{\begin{equation}}
\newcommand{\eeq}{\end{equation}}

\title{Type $1$ and $2$ sets for series of translates of functions}
\author{Zolt\'an Buczolich\thanks{\scriptsize
This author was supported by the Hungarian National Research, Development and Innovation Office--NKFIH, Grant  124749.
},
Department of Analysis, ELTE E\"otv\"os Lor\'and\\
University, P\'azm\'any P\'eter S\'et\'any 1/c, 1117 Budapest, Hungary\\
email: buczo@cs.elte.hu\\
{\tt www.cs.elte.hu/\hbox{$\sim$}buczo}\\
ORCID Id: 0000-0001-5481-8797
  \smallskip\\
  Bruce Hanson, Department of Mathematics,\\ Statistics and Computer Science,\\ St.\ Olaf College,
Northfield, Minnesota 55057, USA\\
{email:} \texttt{hansonb@stolaf.edu}
 \smallskip\\
 Bal\'azs Maga\thanks{\scriptsize This author was supported by the \'UNKP-17-2 New National Excellence of the Hungarian Ministry of Human Capacities, and by the Hungarian National Research, Development and Innovation Office–NKFIH, Grant 124003.},
Department of Analysis, ELTE E\"otv\"os Lor\'and\\
University, P\'azm\'any P\'eter S\'et\'any 1/c, 1117 Budapest, Hungary\\
 email: magab@cs.elte.hu \\{\tt www.cs.elte.hu/\hbox{$\sim$}magab}
  \smallskip\\
and
  \smallskip\\
 G\'asp\'ar V\'ertesy\thanks{\scriptsize This author was supported by the Hungarian National Research, Development and Innovation Office–NKFIH, Grant 124749.
},
 Department of Analysis, ELTE E\"otv\"os Lor\'and\\
University, P\'azm\'any P\'eter S\'et\'any 1/c, 1117 Budapest, Hungary\\
email: vertesy.gaspar@gmail.com\
}
\date{\today}
\begin{document}
\maketitle

\medskip

\medskip
{\em Dedicated to the memory of \'Akos Cs\'asz\'ar}

\medskip


\begin{abstract}
 Suppose $\Lambda$ is a discrete infinite set of nonnegative real numbers.
We say that $ {\Lambda}$ is type $1$ if the series $s(x)=\sum_{\lambda\in\Lambda}f(x+\lambda)$
 satisfies a zero-one law. This means that for any non-negative measurable
 $f: {\ensuremath {\mathbb R}}\to [0,+ {\infty})$ either the convergence set $C(f, {\Lambda})=\{x: s(x)<+ {\infty} \}= {\ensuremath {\mathbb R}}$ modulo sets of Lebesgue zero, or its complement the divergence set $D(f, {\Lambda})=\{x: s(x)=+ {\infty} \}= {\ensuremath {\mathbb R}}$ modulo sets of measure zero.
 If $ {\Lambda}$ is not type $1$ we say that $ {\Lambda}$ is type 2.

The exact characterization of type $1$ and type $2$ sets is not known.
In this paper we continue our study of the properties of type $1$ and $2$ sets.
We discuss sub and supersets of type $1$ and $2$ sets and we give a complete and simple
characterization of a subclass of dyadic type $1$ sets. We discuss the existence
of type $1$ sets containing infinitely many algebraically independent elements.
Finally, we consider unions and Minkowski sums of type $1$ and $2$ sets.
  \end{abstract}


\section{Introduction}\label{*secintro}
This
paper is related to the talk given by the first listed author at the
\'Akos Cs\'asz\'ar  Memorial Conference held at the R\'enyi Institute
on February 26,  2018. In 2017 we lost two outstanding mathematicans
Jean-Pierre Kahane and \'Akos Cs\'asz\'ar.
During the Fall of 2017 in paper \cite{[BMV]}, which was prepared for the Jean-Pierre Kahane
memorial volume of Analysis Mathematica we returned to some open questions
from \cite{[BKM1]}, written by Z. Buczolich, J-P. Kahane and D. Mauldin.
It is a strange recurrence of events that in 1999 at the 75th Birthday conference
of  \'Akos Cs\'asz\'ar the first listed author gave a talk on the results from
\cite{[BKM1]} and now exactly when the continuation of that paper was going on
he had the opportunity to speak about this
topic again at the \'Akos Cs\'asz\'ar  Memorial Conference.

This line of research began with a question which was called
 the
{\it Khinchin conjecture \cite{[Kh]} (1923):}

{\it Assume that
$E {\subset} (0,1)$
is a measurable set and
$f(x)=\chi_{E}(\{x\})$,
where $\{x\}$ denotes the fractional part of $x$.
Is it true that for almost every $x$
$$\frac{1}{k}\sum_{n=1}^{k} f(nx)\to  {\mu}(E)?$$}
(In our paper $ {\mu}$ denotes the Lebesgue measure.)

In 1969
Marstrand \cite{[M]}
proved that the Khinchin conjecture is not true.
Other counterexamples were given by J. Bourgain \cite{[Bou]} by using his entropy method and by A. Quas and M. Wierdl \cite{[QW]}.
For further results related to the Khinchin conjecture  we also refer to
\cite{[Beck]} and \cite{[Bee]} and for some generalizations we mention
\cite{[ABW]}, \cite{[BWa]} and \cite{[BWb]}.

The Khinchin conjecture dealt with periodic functions $f$.
For the non-periodic case
there was a question from 1970, originating
from the Diplomarbeit of
 Heinrich von Weizs\"aker \cite{[HW]}:

{\it Suppose $f:(0,+ {\infty})\to {\ensuremath {\mathbb R}}$ is a measurable function.
Is it true that
$\sum_{n=1}^{{\infty}}f(nx)$
either converges (Lebesgue) almost everywhere or diverges almost everywhere, i.e.
is there  a zero-one law for $\sum f(nx)$?}

This question also appeared in a paper of J. A. Haight \cite{[H1]}.

In  \cite{[BM1]}
the first author and D.~Mauldin gave a negative answer to this question:

\begin{theorem}\label{*HWth} There exists a
measurable function $f:(0,+ {\infty})\to \{0,1\}$ and
two nonempty intervals $I_{F}, \ I_{{\infty}} {\subset} [{1\over 2},1)$
such that for every $x\in I_{{\infty}}$ we have $\sum_{n=1}^{{\infty}}f(nx)=+ {\infty}$
and for almost every $x\in I_{F}$ we have $\sum_{n=1}^{{\infty}}f(nx)<+ {\infty}.$
The function $f$ is the characteristic function of an open set
$E$.\end{theorem}

In papers \cite{[BKM1]} and \cite{[BKM2]}
Z.~Buczolich, J-P.~Kahane and D.~Mauldin
 considered a more general, additive version of the Haight--Weizs\"aker problem.
Since $\sum_{n=1}^{{\infty}}f(nx)=\sum_{n=1}^{{\infty}}f(e^{\log x+\log n})$,
that is using the function $h=f\circ    \exp$ defined on $ {\ensuremath {\mathbb R}}$ and $\Lambda=\{\log n:
n=1,2,... \}$  they were interested in almost everywhere
convergence questions for the series
$\sum_{{\lambda}\in {\Lambda}}h(x+ {\lambda})$.

In  the original ``multiplicative" version of our problem already Haight
in  \cite{[H2]} started to investigate convergence properties
of series $\sum_{{\lambda}\in {\Lambda}}f( {\lambda} x)$.

In this note the symbol $\Lambda$ will always represent a countably infinite, unbounded set of real numbers which is bounded from below and has no finite accumulation points.

Type 1 and type 2 sets were defined in \cite{[BKM1]}.
Given $\Lambda$ and
a measurable $f: {\ensuremath {\mathbb R}}\to [0,+ {\infty})$, we consider the sum
$$s(x)=\sum_{\lambda\in\Lambda}f(x+\lambda),$$
and the complementary subsets of $ {\ensuremath {\mathbb R}}$:
$$C=C(f,\Lambda)=\{x: s(x)< {\infty}\},\qquad
D=D(f,\Lambda)=\{x:s(x)= {\infty}\}.$$

\begin{definition}\label{def1} The set $\Lambda$ is type $1$ if, for every $f$, either
$C(f,\Lambda)= {\ensuremath {\mathbb R}}$ a.e.\!\! or $C(f,\Lambda)= {\emptyset}$ a.e.~(or equivalently
$D(f,\Lambda)= {\emptyset}$ a.e. or $D(f,\Lambda)= {\ensuremath {\mathbb R}}$ a.e.). Otherwise, $\Lambda$
is type $2$.
For type $2$ sets there are non-negative measurable {\it witness functions}
$f$ such that both $C(f,\Lambda)$ and $D(f,\Lambda)$ are of positive measure.
\end{definition}

That is, for type 1 sets we have a ``zero-one" law for the almost everywhere
convergence properties of the series $\sum_{\lambda\in\Lambda}f(x+\lambda)$,
while for type 2 sets the situation is more complicated.

In our recent paper \cite{[BHMV]}, answering a question from \cite{[BKM1]},
we proved the following theorem:

\begin{theorem}\label{char func} Suppose that $\Lambda$ is type $2$, that is there exists a measurable witness function $f$ such that both $D(f,\Lambda)$ and $C(f,\Lambda)$ have positive measure. Then there exists a witness function $g$ which is the characteristic function of an open set and both $D(g,\Lambda)$ and $C(g,\Lambda)$ have positive measure.\end{theorem}

This theorem will be important in this paper as well.

\begin{definition}\label{asydens} The unbounded, infinite  discrete set $\Lambda=\{{\lambda}_{1}, {\lambda}_{2},... \}$, $ {\lambda}_{1}< {\lambda}_{2}<...$ is asymptotically dense if $d_{n}= {\lambda}_{n}- {\lambda}_{n-1}\to 0$, or equivalently:
$$\forall a>0,\quad \lim_{x\to\infty}\#(\Lambda\cap [x,x+a])=\infty.$$
If $d_{n}$ tends to zero monotonically, we speak about
decreasing gap asymptotically dense sets.

If  $\Lambda$ is not asymptotically dense we say that it is asymptotically lacunary.
\end{definition}

We denote by $C^{+}_{0}( {\ensuremath {\mathbb R}})$ the non-negative continuous functions on $ {\ensuremath {\mathbb R}}$  tending to zero in $+ {\infty}$.

By Theorem 4 of \cite{[BKM1]} lacunarity is a sufficient condition for type $2$:
\begin{theorem}\label{*BKMth4}
If $\Lambda$ is asymptotically lacunary, then $\Lambda$ is
type $2$. Moreover, for some $f\in C_{0}^{+}( {\ensuremath {\mathbb R}})$, there exist
intervals $I$ and $J$, $I$ to the left of $J$, such that $C(f,\Lambda)$
contains $I$ and $D(f,\Lambda)$ contains $J.$
\end{theorem}

In \cite{[BKM1]} we gave some necessary and some sufficient conditions
for  a set $ {\Lambda}$ being type 2. A complete characterization of type 2
sets is still unknown. We recall here
from \cite{[BKM1]}
the theorem concerning the Haight--Weizs\"aker problem. This contains the additive version of the result of Theorem \ref{*HWth}
along with some auxiliary information.

\begin{theorem}\label{*BKMHWth} {The set $\Lambda=\{\log n:
n=1,2,... \}$ is type $2$. Moreover, for some $f\in
C_{0}^{+}( {\ensuremath {\mathbb R}}),$ $C(f, {\Lambda})$ has full measure on the half-line
$(0,\infty)$ and $D(f, {\Lambda})$ contains the half-line $(-\infty,0)$.
If for each $c, \int_c^{+\infty}e^yg(y)dy < +\infty$, then $C(g,\Lambda) =  {\ensuremath {\mathbb R}}$ a.e.
If $g\in C_{0}^{+}( {\ensuremath {\mathbb R}})$ and $C(g,\Lambda)$ is not of the first (Baire) category, then
$C(g, {\Lambda})= {\ensuremath {\mathbb R}}$ a.e. Finally, there is some $g\in C_{0}^+( {\ensuremath {\mathbb R}})$ such that $C(g,\Lambda) =  {\ensuremath {\mathbb R}}$ a.e. and
$\int_0^{+\infty}e^yg(y)dy = +\infty$.
}
\end{theorem}

One might believe that for type 2 sets $\Lambda$ the sets $C(f, {\Lambda})$, or $D(f, {\Lambda})$ are always half-lines
if they differ from $ {\ensuremath {\mathbb R}}$.
Indeed in \cite{[BKM1]} we obtained results in this direction.
A number $t>0$ is called a translator of $ {\Lambda}$ if $( {\Lambda}+t) {\setminus}
 {\Lambda}$ is finite.
Condition $(*)$ is said to be satisfied if $T( {\Lambda})$, the
countable additive
semigroup of translators of $ {\Lambda}$, is dense in $ {\ensuremath {\mathbb R}}^{+}$.
We recall Proposition 3 of \cite{[BKM1]}:

\begin{proposition}\label{*proptrans}
Suppose that condition $(*)$ is
satisfied ($\Lambda$ has arbitrarily small translators).
Then the topological closure of $C$ (resp. $D$) is either $ {\emptyset}$, or
$ {\ensuremath {\mathbb R}}$, or else a closed right half-line (resp. left half-line). The same
holds for the support of $\pmb{1}_{C}$ (resp. $\pmb{1}_{D}$) meaning the
smallest closed set $\pmb{C}$ carrying $C$ (resp. $\pmb{D}$ carrying $D$) except for a null set.
The interior of $\pmb{C}$ (resp. $\pmb{D}$) is either $ {\emptyset}$, or $ {\ensuremath {\mathbb R}}$, or else an
open right (resp. left) half-line.
\end{proposition}

Determining the structure of convergence and divergence sets for type $2$ sets
is an interesting problem. In the recent paper \cite{[BMV]} we proved
the following theorem:

\begin{theorem}\label{*thdflg}
There is a strictly monotone increasing unbounded sequence $(\lambda_0,\lambda_1,\ldots)=\Lambda$ in $ {\mathbb {R}}$ such that  $\lambda_{n}-\lambda_{n-1}$ tends to $0$ monotonically,
that is $ {\Lambda}$ is a decreasing gap asymptotically dense set,
such that for every open set $G\subset {\mathbb {R}}$ there is a function $f_G: {\ensuremath {\mathbb R}}\to [0,+ {\infty})$ for which
\begin{equation}\label{tetelbeli}
\mu\left(\left\{x\notin G :  \sum_{n=0}^\infty  f_G(x+\lambda_n)=\infty\right\}\right)=0, \text{    and    }
 \end{equation}
\begin{equation}\label{*conv}
 \text{    $\sum_{n=0    }^\infty f_G(x+\lambda_n) =\infty$ for every $x\in G$,}
\end{equation}
moreover
$f_G=\chi_{U_G}$ for a closed set $U_G\subset {\mathbb {R}}$.
By \eqref{tetelbeli} and \eqref{*conv} we have $D(f_G, {\Lambda})\supset G$, and
$C(f_G, {\Lambda})= {\ensuremath {\mathbb R}} {\setminus} G$ modulo sets
of measure zero.

One can also select a $g_G\in C_{0}^{+}( {\ensuremath {\mathbb R}})$ satisfying \eqref{tetelbeli} and
\eqref{*conv} instead of $f_G$.
\end{theorem}

In this paper two examples from \cite{[BKM1]}, quoted in this paper as
Examples \ref{*exdyada}
and  \ref{*exdyadb}
will play
an important role:

\begin{example}\label{*exdyada}
Set $\Lambda=\cup_{k \in {{{\ensuremath {\mathbb N}}}}}\Lambda_{k},$ where $\Lambda_{k}=2^{-
k} {{{\ensuremath {\mathbb N}}}}\cap [k,{k+1}).$
In Theorem 1 of \cite{[BKM1]} it is proved that $\Lambda$ is type $1$.
In fact, in a slightly more general version it is shown that
if $(n_{k})$ is an increasing sequence of
positive integers  and $\Lambda=\cup_{k\in  {{{\ensuremath {\mathbb N}}}}}\Lambda_{k}$ where $\Lambda_{k}=2^{-
k} {{{\ensuremath {\mathbb N}}}}\cap [n_{k},n_{k+1})$
then $\Lambda$ is type $1$.
\end{example}

In \cite{[BHMV]} we studied the effect of randomly deleting elements
of $ {\Lambda}$.
Let $0<p<1$.  Then we say that $\Lambda\subset \widetilde{\Lambda}$ is chosen with probability $p$ from $\widetilde{\Lambda}$ if for each $\lambda \in \widetilde{\Lambda}$ the probability that $\lambda \in \Lambda$ is $p$.
Let
$
\widetilde{\Lambda}=\bigcup_{k=1}^\infty (2^{-k} {\mathbb {N}} \cap[k,k+1)).
$
 We know from
Example \ref{*exdyada} that $ {{\widetilde {\Lambda}}}$ is type $1$.
By Theorem 4.3 of \cite{[BHMV]}
if $\Lambda$ is chosen with probability $p$ from $\widetilde{\Lambda}$
then almost surely $\Lambda$ is type $1$.

However as Theorem 4.5 of \cite{[BHMV]} shows, it may happen that
type $1$ sets are converted into type $2$ sets by random deletion:
\begin{theorem}\label{*randt2} Suppose that $(m_k)$ and $(n_k)$ are strictly increasing sequences of
positive integers. For each $k\in\mathbb{N}$, define $\Lambda_k=2^{-m_k} \mathbb{N} \cap [n_k, n_{k+1})$ and let $ {{\widetilde {\Lambda}}}=\bigcup_{k=1}^{\infty}\Lambda_k$. Moreover, fix $0<p<1$ and suppose that $\Lambda$ is chosen  with probability $p$ from $ {{\widetilde {\Lambda}}}$.
Set $q=1-p$.
 For fixed $(m_k)$, if $(n_k)$ tends to infinity sufficiently fast
 then  almost surely $ {\Lambda}$ is type $2$. Notably, if the series $\sum_{k=1}^{\infty}1-\left(1-{q}^{2^{m_k}}\right)^{n_{k+1}-n_k}$ diverges then  almost surely $ {\Lambda}$ is type $2$. \end{theorem}

\begin{example}\label{*exdyadb}
Let
$(n_{k})$ be a given increasing sequence of positive integers.
 By Theorem 3 of \cite{[BKM1]}
there is
an increasing sequence of integers $(m(k))$ such that the set
$\Lambda=\cup_{k\in  {{{\ensuremath {\mathbb N}}}}}\Lambda_{k}$ with $\Lambda_{k}=2^{-m(
k)} {{{\ensuremath {\mathbb N}}}}\cap [n_{k},n_{k+1})$ is type $2$.
\end{example}

Given $x \in  {\mathbb R}$ and a set $A \subset  {\mathbb R}$  we define $x+A=\{x+a: \ a\in A\}$ and for $y \in  {\mathbb R}$ we define
$y-A=\{y-a : \ a\in A\}$.  Similarly for sets $A,B \subset  {\mathbb R}$ we define
$A+B=\{a+b: \  a\in A,\  \, b \in B\}$ and $A-B=\{a-b: \  a\in A, \  \, b \in B\}$.

According to Theorem 6 of \cite{[BKM1]},
type $2$ sets  form a dense
open subset in the box topology of discrete sets while type $1$ sets  form a closed
nowhere dense set. Therefore type $2$ is typical in the Baire
category sense in this topology.
This also shows that it is usually  more difficult to find and verify type $1$ sets.
The question of complete characterization of type $1$ and type $2$
sets is a difficult and unsolved problem. The goal of this paper is to explore
some properties of these sets and provide some more examples of type $1$
and type $2$ sets.

In Theorem 5 of \cite{[BKM1]} we obtained a sufficient condition for type 2 based on
independent elements in $ {\Lambda}$. This is the following result:

\begin{theorem}\label{*BKMth5}
Suppose that there exist three intervals
$I$, $J$, $K$ such that $J=K+I-I$, the interval
$I$ is to the left of $J$, and $\hbox{dist}(I,J)\geq |I|,$ and
two sequences $(y_{j})$
and $(N_{j})$ tending to infinity ($y_{j}\in  {\ensuremath {\mathbb R}}^{+}$,
$N_{j}\in {\ensuremath {\mathbb N}}$) such that, for each $j$, $y_{j}-I$
contains a set of $N_{j}$ points of $\Lambda$
independent from $\Lambda \cap (y_j-J)$ in the sense that
the additive groups generated by these sets
have only $0$ in common. Then $\Lambda$ is type
$2$. Moreover, for some $f\in C_{0}^{+}( {\ensuremath {\mathbb R}})$, $D(f,\Lambda)$ contains $I$ and
$C(f,\Lambda)$ has full measure on $K$.
\end{theorem}

In \cite{[BKM1]} we showed that $\Lambda=\{\log n:
n=1,2,... \}$ is type $2$ by using Theorem \ref{*BKMth5}.

Recall that the set
$\{\alpha_1,\alpha_2,\dots\}$ consists of algebraically independent
numbers if for each $N \in  {\mathbb {N}}$ if $k_1,k_2,\dots,k_N \in  {\mathbb {Z}}$ and $k_1\alpha_1+k_2\alpha_2+\dots +k_N\alpha_N=0$, then
$k_1=k_2=\dots=k_N=0$.

We also recall part of the remark following Theorem 5 in \cite{[BKM1]}:
\begin{remark}\label{*rembkm5}
If $\Lambda$ is asymptotically dense and consists of elements
independent over $ {\ensuremath {\mathbb Q}}$ then using Theorem \ref{*BKMth5} it is easy to show that
$\Lambda$ is type $2.$
\end{remark}

In \cite{[BM1]} it was established that $\Lambda=\{\log n:
n=1,2,... \}$ is type 2 via a corollary of Kronecker's Theorem \cite[p. 53]{[C]}:
\begin{theorem}\label{*kron}
 Assume $\theta_{1},...,\theta_{L}\in {\ensuremath {\mathbb R}}$ and
$(\alpha_{1},...,\alpha_{L})$ is a real vector. The following two statements
are equivalent:
\item{A)} For every $ \epsilon>0$, there exists $p\in {\ensuremath {\mathbb Z}}$ such that
$$||\theta_{j}p-\alpha_{j}||< \epsilon, \hbox{ for } 1\leq j\leq L,$$
where $||x||=\min \{|x-n|: n\in {\ensuremath {\mathbb Z}}\}$.
\item{B)} If $(u_{1},...,u_{L})$ is a vector consisting of integers and
$$u_{1}\theta_{1}+...+u_{L}\theta_{L}\in {\ensuremath {\mathbb Z}},$$
then $$u_{1}\alpha_{1}+...+u_{L}\alpha_{L}\in {\ensuremath {\mathbb Z}}.$$
\end{theorem}

This paper is organized in the following way:
In Section \ref{secsubsuper} we begin with Theorem \ref{ez is 2-es} which gives a
sufficient condition for a set to be type $2$ by saying that if the cardinality of $ {\Lambda}$
in subsequent intervals increases with sufficiently large jumps then $ {\Lambda}$
is type $2$. As an application of this theorem
in Theorem \ref{mkthm} we obtain a complete characterization of type $1$
and type $2$ sets which are defined analogously to Examples \ref{*exdyada}
and \ref{*exdyadb}.
Corollary \ref{suru 2-es}
is an immediate consequence of Theorem \ref{ez is 2-es}
and gives an example of a type $2$ set, $ {\Lambda}$ such that any $ {\Lambda}'\supset  {\Lambda}$
is also type $2$. In Theorem \ref{slow always 2} we show that the growth rate assumption given in Corollary  \ref{suru 2-es} can be significantly relaxed
in the case where $\Lambda$ contains sufficiently many
algebraically independent elements.
In Theorem \ref{always type 2} we give an example
of a $ {\Lambda}$ such that every infinite subset of $ {\Lambda}$ and every superset of $ {\Lambda}$ is type $2$.
In Theorem \ref{sandwich} we see that we can have a bi-infinite nested sequence
$ {\Lambda}_{n+1} {\subset}  {\Lambda}_{n}$, $n\in  {\ensuremath {\mathbb Z}}$  such that $ {\Lambda}_{n}$ is type $1$ for odd $n$
and type $2$ for even $n$.

Before writing this note we were not aware of any type $1$ sets
 containing infinitely many algebraically
independent elements and Theorem \ref{*BKMth5} also suggests that
many independent elements lead to type $2$ sets.
This is illustrated by Theorem \ref{type1plusalg} which roughly states
that if we add an infinite set of algebraically independent numbers to a  set from Examples \ref{*exdyada}
and \ref{*exdyadb} to obtain a discrete $ {\Lambda}$ then we always obtain type $2$ sets.
On the other hand, in Theorem \ref{alg indep type 1} we see that there exist
type $1$ sets which contain infinitely many algebraically independent
numbers.

In Section \ref{secminko} we consider unions and Minkowski sums.
From Proposition  \ref{*proptos} we see that unions of type $1$ sets are always type $1$,
while in Proposition \ref{2-esek unioja} we prove that it may happen that the union of two
decreasing gap asymptotically dense type $2$
sets is type $1$.
We see in Theorem \ref{*minksum} that Minkowski sums of type $1$ sets are type $1$.
Finally, in Theorem \ref{*mink2} we prove that there is a type $2$ set $ {\Lambda}$
such that for any infinite discrete $ {\Lambda}'$ the
Minkowski sum set $ {\Lambda}+ {\Lambda}'=\{\lambda+\lambda':\, \lambda\in  {\Lambda},\  \lambda'\in {\Lambda}' \}$ is type $2$.
On the other hand, simple examples show that 
 it may happen that the Minkowski sum of two type $2$ sets is type $1$.


\section{Sub and supersets of type 1 and 2 sets}
\label{secsubsuper}

\begin{theorem}\label{ez is 2-es}
Let $ {\varepsilon}$ be a positive number. 
For every $n\in {\ensuremath {\mathbb Z}}$ we denote the cardinality of $\Lambda\cap [n {\varepsilon},(n+1) {\varepsilon})$ by $a_n$. If
\begin{equation}\label{gyors}
\limsup_{n\rightarrow\infty} \frac{a_n}{a_{n-1}} = \infty
\end{equation}
(where $\frac{0}{0}=0$ and $\frac{c}{0}=\infty$ if $c>0$), then $\Lambda$ is type 2.
\end{theorem}
\begin{proof}
Let $ {\varepsilon}':=\frac{{\varepsilon}}{3}$ and $a'_n:=\#(\Lambda\cap [n {\varepsilon}',(n+1) {\varepsilon}'))$ for every $n\in {\ensuremath {\mathbb N}}$. We will prove that
\begin{equation}\label{gyors_2-es 2}
\limsup_{n\in {\ensuremath {\mathbb N}}} \frac{a'_n}{a'_{{n}-3}+a'_{n-2}+a'_{n-1}} = \infty.
\end{equation}

Proceeding towards a contradiction suppose that there exists a positive number $c$ and $N_{0}\in  {\ensuremath {\mathbb N}}$ such that for every  $n\geq N_{0}$ we have
$$
\frac{a'_n}{a'_{n-3}+a'_{n-2}+a'_{n-1}} < c.
$$
Thus, if $a_{n-1}>0$ and $n\geq N_{0}$ then
\begin{align*}
\frac{a_n}{a_{n-1}} &= \frac{a'_{3n}+a'_{3n+1}+a'_{3n+2}}{a_{n-1}} \leq \frac{a'_{3n}+a'_{3n+1}+c\left(a'_{3n-1}+a'_{3n}+a'_{3n+1}\right)}{a_{n-1}} = \\
&= \frac{ca'_{3n-1}+(c+1)a'_{3n}+(c+1)a'_{3n+1}}{a_{n-1}} \leq  \\
&\leq \frac{ca'_{3n-1}+(c+1)a'_{3n}+(c+1)c\left(a'_{3n-2}+a'_{3n-1}+a'_{3n}\right)}{a_{n-1}} \\
&\leq \frac{ca_{n-1}+(c+1)ca_{n-1}+(c+1)c\left(a_{n-1}+a_{n-1}+ca_{n-1}\right)}{a_{n-1}} = \\
&= c+(c+1)c+(c+1)c(1+1+c),
\end{align*}
which contradicts \eqref{gyors}.

We can assume that $ {\varepsilon}'=1$ since $\Lambda$ and $\frac{1}{{\varepsilon}'}\Lambda$ have the same type.

We construct a function $f$ such that $[0,1)\subset C(f,\Lambda)$ and $[-2,-1)\subset D(f,\Lambda)$. We choose a sequence $\left(m_k\right)$ in $ {\ensuremath {\mathbb N}}$ for which
\begin{equation}\label{a'_{m_k} nagy}
\frac{a'_{m_k}}{a'_{{m_k}-3}+a'_{{m_k}-2}+a'_{{m_k}-1}} \ge 2^k \text{  and  } m_{k+1}-m_k\ge 2
\end{equation}
for every $k\in  {\ensuremath {\mathbb N}}$, and we set $f=\left(a'_{m_k}\right)^{-1}$ on $[m_k-2,m_k)$. Everywhere else let $f=0$.
Then for any $x$ we have
\begin{displaymath}
\sum_{\lambda\in\Lambda}f(x+\lambda)=\sum_{k=1}^{\infty}\quad\sum_{\lambda \in [m_k-2-x,m_k-x)\cap\Lambda} f(x+\lambda).
\end{displaymath}
In order to prove the claim we need to estimate the sum
$$\sum_{\lambda \in [m_k-2-x,m_k-x)\cap\Lambda} f(x+\lambda)$$
for each $k\in\mathbb{N}$, and $x\in[-2,-1)$, or $x\in[0,1)$. First of all, we note that if $x\in[0,1)$, then $x+\lambda\in[m_k-2,m_k)$ implies $\lambda\in[m_k-3,m_k)$, and in this interval $\Lambda$ has $a'_{m_k-1}+a'_{m_k-2}+a'_{m_k-3}$ elements
and $f=\left(a'_{m_k}\right)^{-1}$ on $[m_{k}-2,m_{k})$. As a consequence, by \eqref{a'_{m_k} nagy} we have
\begin{displaymath}
\sum_{\lambda \in [m_k-2-x,m_k-x)\cap\Lambda} f(x+\lambda) \le \left(a'_{m_k-1}+a'_{m_k-2}+a'_{m_k-3}\right)\left(a'_{m_k}\right)^{-1} \le \frac{1}{2^k}
\end{displaymath}
for any $x\in[0,1)$. However, the series $\sum\frac{1}{2^k}$ converges which yields $[0,1)\subseteq C(f,\Lambda)$.

For the other containment we simply notice that the number of terms in
$$
\sum_{\lambda \in [m_k-2-x,m_k-x)\cap\Lambda} f(x+\lambda)
$$
for $x\in[-2,-1)$ is at least $a'_{m_k}$ as $x+\lambda \in [m_k-2,m_k)$ for every $\lambda\in[m_k,m_k+1)$.
 Hence  we obtain
$$
\sum_{\lambda \in [m_k-2-x,m_k-x)\cap\Lambda} f(x+\lambda) \geq a'_{m_k}\left(a'_{m_k}\right)^{-1} = 1
$$
for $x\in[-2,-1)$. As the series $\sum 1$ diverges it follows that $[-2,-1)\subseteq D(f,\Lambda)$. This concludes the proof.
\end{proof}

\begin{corollary}\label{suru 2-es}
For every $n\in {\ensuremath {\mathbb Z}}$ we denote the cardinality of $\Lambda\cap [n,n+1)$ by $a_n$. If
\begin{equation}\label{fast}
\limsup\limits_{n\rightarrow\infty}\frac{a_n}{c^n}=\infty \text{  for every positive $c\in {\ensuremath {\mathbb R  }}$,}
\end{equation}
then $\Lambda\subset\Lambda'$ implies that $\Lambda'$ is type 2.
\end{corollary}
\begin{proof}
If $\Lambda\subset\Lambda'$ then $\Lambda'$ also satisfies \eqref{fast}, hence it is enough to prove that $\Lambda$ is type 2.

We will use Theorem \ref{ez is 2-es} with $ {\varepsilon}:=1$. If \eqref{gyors} were not
satisfied, there would be some $c\in {\ensuremath {\mathbb R}}$ such that
$$
\limsup_{n\rightarrow\infty} \frac{a_n}{a_{n-1}} < c,
$$
which contradicts \eqref{fast}.
\end{proof}

The next theorem shows that Corollary \ref{suru 2-es} is sharp in some sense.
Later in Theorem \ref{slow always 2} we prove that assumption \eqref{fast}
can be significantly relaxed for algebraically independent numbers
and the converse of Corollary \ref{suru 2-es} is not true.

Theorem \ref{mkthm} is a much sharper version of Example \ref{*exdyadb} since it gives a
necessary and sufficient condition for a set obtained by the ``dyadic" construction
being type $1$ (or type $2$).

\begin{theorem}\label{mkthm}
Suppose that $(m_k)$ and $(n_k)$ are strictly increasing sequences of positive integers.  For each $k \in  {\mathbb {N}}$, define $\Lambda_k=2^{-m_k} {\mathbb {N}} \cap [n_k,n_{k+1})$ and let $\Lambda=\cup_{k=1}^\infty \Lambda_k$.  Define $  \displaystyle  M=\sup_k\{m_{k+1}-m_k\}$.  Then $\Lambda$ is type 1 if and only if $M < \infty$.

\end{theorem}

\begin{proof}  Assume that $M < \infty$.  In this case a straightforward modification of the proof of Theorem 1 in \cite{[BKM2]} shows that $\Lambda$ is type 1. Suppose that $M=\infty$. Now we can use Theorem \ref{ez is 2-es} with $ {\varepsilon}:=1$. We have
$$
\limsup_{k\rightarrow\infty} \frac{a_{n_k}}{a_{n_k-1}} = \limsup_{k\rightarrow\infty} 2^{m_k-m_{k-1}} = \infty,
$$
hence $\Lambda$ is type 2.
\end{proof}

\begin{theorem}{\label{sandwich}}
There exists a collection of discrete sets $\{\Lambda_n\}_{n \in  {\mathbb {Z}}}$ such that
 \begin{equation}\label{sandwichineq}
\Lambda_{n+1} \subset \Lambda_{n} \text{    for all    } n \in  {\mathbb {Z}}
\end{equation}
and $\Lambda_n$ is type 1 if $n$ is odd and type 2 if $n$ is even.
\end{theorem}

\begin{proof} For each $k \in  {\mathbb {Z}}$ define $\Gamma_k=2^{-k} {\mathbb {N}}$ and for each $j\in  {\mathbb {Z}}$ and $\nu \in  {\mathbb {N}}$ define
$l(\nu,j)= \lfloor \nu\cdot 2^{-j} \rfloor$ and $m(\nu,j)=\max\{2^i\colon\, 2^i\le \nu\cdot 2^{-j}\}$.  Note that for every $j \in  {\mathbb {Z}}$ we have
 \begin{equation}\label{lm ineq}
l(\nu,j+1)\le m(\nu,j) \le l(\nu,j),
\end{equation}
 \begin{equation}\label{finite}
\sup_{\nu \in  {\mathbb {N}}}(l(\nu+1,j)-l(\nu,j))= \lceil 2^{-j}  \rceil < \infty,
\end{equation}
and
 \begin{equation}\label{infinite}
\sup_{\nu \in  {\mathbb {N}}}(m(\nu+1,j)-m(\nu,j))= \infty.
\end{equation}

  For each $j \in  {\mathbb {Z}}$ we define
$\Lambda_{2j}=\cup_{\nu=1}^\infty (\Gamma_{m(\nu,j)}\cap [\nu,\nu+1))$ and we define $\Lambda_{2j-1}=\cup_{\nu=1}^\infty (\Gamma_{l(\nu,j)}\cap [\nu,\nu+1))$.  Then \eqref{sandwichineq} follows directly from \eqref{lm ineq}.
 By Theorem \ref{mkthm}, \eqref{finite} and \eqref{infinite} we see that for every $j \in  {\mathbb {Z}}$ we have that $\Lambda_{2j-1}$ is type $1$ and $\Lambda_{2j}$ is type $2$. \end{proof}


\begin{theorem}\label{alg indep}
Suppose that $\Lambda=\{\alpha_1,\alpha_2,\dots\}$, where $\alpha_n \to \infty$ and the $\alpha_i$s are algebraically independent.  Then $\Lambda$ is type 2.
\end{theorem}

\begin{proof}  If $\Lambda$ is lacunary this follows from Theorem \ref{*BKMth4}.  Otherwise, we can use Remark \ref{*rembkm5}.
\end{proof}

\begin{theorem}\label{slow always 2}
If $\Lambda$ is a discrete infinite set of algebraically independent numbers and
\begin{equation}\label{speed}
\limsup_{n\rightarrow\infty} \frac{\#(\Lambda\cap[0,n))}{n} = \infty
\end{equation}
then every set containing $\Lambda$ is type 2.
\end{theorem}
\begin{proof}
Suppose that $\Lambda\subset \Lambda'$. We will prove that $\Lambda'$ satisfies the conditions of Theorem \ref{*BKMth5}, which implies that it is type 2. Define
$$
\Lambda^* = \left\{\lambda'\in\Lambda' : \text{  $\lambda'$ is independent from $\Lambda'\cap(-\infty,\lambda')$  } \right\}.
$$
It is easy to see that \eqref{speed} is true for $\Lambda^*$ as well. Hence there is a sequence $(y_j)$ in $ {\ensuremath {\mathbb N}}$ tending to infinity such that $N_j := j \le \#(\Lambda^*\cap[y_j,y_j+1))$. Let $I=(-1,0]$, $K=[2,3]$ and $J=K+I-I=(1,4)$.
 By the definition of $\Lambda^*$ we have that $(\Lambda^*\cap(y_j-I)) = (\Lambda^*\cap[y_j,y_j+1)) {\subset}  {\Lambda}'$ is independent from $\Lambda'\cap (y_j-4,y_j-1) = \Lambda'\cap(y_j-J)$, thus $\Lambda'$, $I$, $J$, $K$, $(y_j)$ and $(N_j)$ indeed satisfy the conditions of Theorem \ref{*BKMth5}.
\end{proof}

 One may wonder if it is always possible to construct
a chain appearing in Therem \ref{sandwich}
such that $\Lambda_{0}$ is an arbitrary type $2$ set.
Combining the previous two theorems we obtain a negative answer:

\begin{theorem}\label{always type 2} Assume that $\Lambda$
satisfies
\eqref{speed}
 and $\Lambda$ consists of algebraically independent numbers.
 In this case for any $\Lambda'$ satisfying $\Lambda'\subseteq\Lambda$ or $\Lambda\subseteq\Lambda'$ we have that $\Lambda'$ is type 2. \end{theorem}

\begin{proof} The claim about sets contained by $\Lambda$ is obvious from Theorem \ref{alg indep}.
 As $\Lambda$ satisfies \eqref{speed}, we obtain from Theorem
\ref{slow always 2} that every $\Lambda'$ containing $\Lambda$ is type 2.
\end{proof}

\section{Independent elements}\label{independent}



\begin{theorem}\label{type1plusalg}
Let $\{m_k\}_{k=1}^N$ and $\{n_k\}_{k=1}^N$ be strictly increasing sequences of positive integers, where either $N \in  {\mathbb {N}}$,
 or $N=\infty$.  If $N \in  {\mathbb {N}}$ we define $n_{N+1}=\infty$. Define
$$\Lambda_1=\cup_{k=1}^N (2^{-m_k} {\mathbb {N}} \cap [n_k,n_{k+1})).$$
Let $\Lambda_2=\{\alpha_k:\ k\in {\mathbb {N}}\}$ be an algebraically independent set of irrational numbers, where $\alpha_k \nearrow \infty$.
Then $\Lambda_*=\Lambda_1 \cup \Lambda_2$ is type 2.
\end{theorem}

Observe that $ {\Lambda}_{1}$ can be any of the sets from Examples \ref{*exdyada}
or \ref{*exdyadb}, hence $ {\Lambda}_{1}$ can be a type $1$ set which is converted in this case
into a type $2$ set after we add the independent numbers.

\begin{proof}
 We assume that $0<\alpha_1<\alpha_2<\dots$ and choose a subsequence $\{\alpha_{\nu_k}\}:=\{\beta_k\}$ such that
 \begin{equation}\label{spacing}
\beta_{k+1}-\beta_k>5 \quad \forall k \in  {\mathbb {N}}.
\end{equation}
For each $k \in  {\mathbb {N}}$ we define$$\text{    
$B_k=\{\beta_{2^k+1},\beta_{2^k+2},\dots,\beta_{2^{k+1}}\}$ and
$A_k=\{\alpha_1,\alpha_2,\dots,\alpha_{\nu_{2^{k+1}}}=\beta_{2^{k+1}}\}$.}$$
 We also define
$$r_k=\sup\{m_l: \ n_l \le \beta_{2^{k+1}}+1\},$$
and let $A_k^*=\{2^{r_k}\alpha:\ \alpha \in A_k \}$.

For each $i=1,2,\dots, 2^k$ we define
$$\alpha_{i,k}=\beta_{2^k+i}=\alpha_{\nu_{2^k+i}}.$$

For each $k \in  {\mathbb {N}}$ we let $\mathbf{n}=\mathbf{n}(k)=2^k$.   Note that $A^*_k$ is a finite set of algebraically independent numbers and therefore using Kronecker's Theorem (Theorem \ref{*kron}) we may choose
$p_k\in  {\mathbb {N}}$ such that

 \begin{equation}\label{alg i/n2}
\Big |\Big |p_k 2^{r_k} \alpha_{i,k}+\frac{i}{\mathbf{n}}\Big |\Big |\le \frac1{10\mathbf{n}},\quad \ \ i=1,2,\dots,\mathbf{n}
\end{equation}
and
 \begin{equation}\label{alg 02}
\Big |\Big |p_k 2^{r_k} \alpha_j\Big |\Big |\le \frac1{10\mathbf{n}} \text{     for    }\alpha_j \in A_k \backslash B_k.
\end{equation}
%
%

We also define $t_k=p_k 2^{r_k}$ and for each $i \in \{1,2,\dots,2^k\}$ and $k \in  {\mathbb {N}}$ we define
$$ S_{i,k}=\Big (\bigcup_{j=1}^\infty \Big [\frac{j}{t_k}-\frac1{\mathbf{n}t_k},\frac{j}{t_k}+\frac1{\mathbf{n}t_k}
\Big ]\Big )\cap [\alpha_{i,k},\alpha_{i,k}+1]$$
and then define
$$ S =\cup_{k=1}^\infty \cup_{i=1}^{2^k} S_{i,k}=\cup_{k=1}^\infty S_k$$
and $f=\mathbf{1}_S$.
Note that by \eqref{spacing} we have
 \begin{equation}\label{sk spacing}
\text{dist}(S_{i,k},S_{i',k'})\ge 4 \quad\text{     if    } (i,k) \neq (i',k').
\end{equation}

\bigskip

\begin{claim}\label{1434}
$[\frac14,\frac34]\subset D(f,\Lambda_{*})$.
\end{claim}

\begin{proof}[Proof of Claim \ref{1434}]
 Let $x \in [\frac14,\frac34]$ and $k \in  {\mathbb {N}}$ and recall that $\mathbf{n}=2^k$. Choose $i\in  {\mathbb {N}}$ and $l_k\in \{1,2,\dots,2^k\}$ such that
 \begin{equation}\label{approx x}
\Big |x-\frac{i}{t_k}-\frac{l_k}{\mathbf{n}t_k}\Big |\le \frac1{2\mathbf{n}{t_k}}.
\end{equation}
We will show that $x+\alpha_{l_k,k}\in S_{l_k,k}\subset S$.  Note that by \eqref{alg i/n2} we have
 \begin{equation}\label{alpha close}
\Big |\Big |t_k \alpha_{l_k,k}+\frac{l_k}{\mathbf{n}}\Big |\Big |\le \frac1{10\mathbf{n}}.
\end{equation}
Thus, we can choose $i'\in  {\mathbb {Z}}$ such that
 \begin{equation}\label{alpha closer}
\Big |t_k\alpha_{l_k,k}-i'+\frac{l_k}{\mathbf{n}}\Big |\le \frac1{10\mathbf{n}}
\end{equation}
and therefore
 \begin{equation}\label{alpha closest}
\Big |\alpha_{l_k,k}-\frac{i'}{t_k}+\frac{l_k}{\mathbf{n}t_k}\Big |\le \frac1{10\mathbf{n}t_k}.
\end{equation}

It follows from \eqref{approx x} and \eqref{alpha closest} that
 \begin{equation}\label{x+alpha}
\Big |x+\alpha_{l_k,k}-\frac{i+i'}{t_k}\Big |<\frac1{\mathbf{n}t_k}.
\end{equation}

Since $x \in [\frac14,\frac34]$, it follows that $x+\alpha_{l_k,k}\in [\alpha_{l_k,k},\alpha_{l_k,k}+1]$ and therefore $x+\alpha_{l_k,k}\in S_{l_k,k}\subset S$. Thus,
$$\sum_{i=1}^\infty f(x+\alpha_i) \ge \sum_{k=1}^\infty f(x+\alpha_{l_k,k})=\infty.$$
Since $x$ was chosen arbitrarily from $[\frac14,\frac34]$, it follows that $[\frac14,\frac34]\subset D(f,\Lambda_2)\subset D(f,\Lambda_*)$, as claimed.
\end{proof}
\bigskip

\begin{claim}\label{34}
 $\mu(C(f,\Lambda_*)\cap [3,4])=1$.
 \end{claim}

\begin{proof} [Proof of Claim \ref{34}] For each $k \in  {\mathbb {N}}$ let
$$F_k=[3,4] \cap \Big (\bigcup_{j\in  {\mathbb {N}}} \Big [\frac{j}{t_k}-\frac{3}{2\mathbf{n}t_k},\frac{j}{t_k}+\frac{3}{2\mathbf{n}t_k}\Big ]\Big ),$$
$$\Lambda_{i,k}=\{\alpha \in \Lambda_2: \ ([3,4]+\alpha) \cap S_{i,k}\neq \emptyset\},$$
$$\mathbf{\Lambda}_k=\cup_{i=1}^{2^k} \Lambda_{i,k},$$
and define
$$ \pmb{s}_k(x) = \sum_{\alpha \in \mathbf{\Lambda}_k} f(x+\alpha).$$
Note that
 $$\sum_{\alpha \in \Lambda_2} f(x+\alpha)=\sum_{k=1}^\infty \pmb{s}_k(x).$$
 Let $D_k=\{x \in [3,4] :\ \pmb{s}_k(x)>0 \}$.
 We claim that $D_k \subset F_k$.
 For each $i=1,2,\dots,2^k$ let $$D_{i,k}=\{x \in [3,4] : \ x + \alpha \in S_{i,k} \text{     for some    }\alpha \in \Lambda_{i,k}\}.$$
Note that $D_k=\cup_{i=1}^{2^k} D_{i,k}$.     Let $x \in D_k$.  Then choose $i$ such that  $x \in D_{i,k}$ and choose $\alpha \in \Lambda_{i,k}$ such that $x+\alpha \in S_{i,k}$.
It follows
that $\alpha_{i,k}-4\leq  {\alpha}\leq  {\alpha}_{i,k}-2$
and thus by
 \eqref{sk spacing} we have that $\alpha \in A_k \backslash B_k$.
Therefore, by \eqref{alg 02} we have
 \begin{equation}\label{alg 1}
||t_k\alpha||\le \frac1{10\mathbf{n}},
\end{equation}
where $\mathbf{n} =2^k$. Thus we can choose $j \in  {\mathbb {N}}$ such that $|t_k \alpha-j|\le \frac1{10\mathbf{n}}$, and hence
 \begin{equation}\label{alpha frac}
\Big |\alpha-\frac{j}{t_k}\Big |\le \frac1{10\mathbf{n}t_k}.
\end{equation}
Moreover, since $x+\alpha \in S_{i,k}$ we can choose $j'\in  {\mathbb {N}}$ such that
 \begin{equation}\label{x frac}
\Big |x+\alpha-\frac{j'}{t_k}\Big | \le \frac1{\mathbf{n} t_k}.
\end{equation}
From \eqref{alpha frac} and \eqref{x frac} we
deduce that $|x-\frac{j'-j}{t_k}| \le \frac1{\mathbf{n}t_k}+\frac1{10\mathbf{n}t_k}<\frac3{2\mathbf{n}t_k}$
 and hence we can conclude that $x \in F_k$. It follows that
$\mu(D_k) \le \mu(F_k) \le \frac1{2^{k-2}}$ and thus $\sum_{k=1}^{{\infty}} \mu(D_k) < \infty$. Thus by the Borel--Cantelli Lemma we obtain
$$\mu\Big (\Big \{x \in [3,4] : \ \sum_{k=1}^\infty \pmb{s}_k(x)=\infty \Big \}\Big )=0$$ and therefore
$\mu(C(f,\Lambda_2)\cap [3,4])=1$.

To complete the proof of Claim \ref{34} we need to show that
$\mu(C(f,\Lambda_1)\cap [3,4])=1$.

Define
$$\Lambda^1_{k}=\{\lambda \in \Lambda_1: \ ([3,4]+\lambda) \cap S_{k}\neq 0\}$$
and
$$ u_k(x) = \sum_{\lambda \in \Lambda^1_k} f(x+\lambda),$$
so $\sum_{\lambda \in \Lambda_1} f(x+\lambda)=\sum_{k=1}^\infty u_k(x).$  Also, we define
$$E_k=\{x \in [3,4] : \ x+\lambda \in S_k \text{     for some    }\lambda \in \Lambda_1\}.$$ Note that
$E_k=\{x\in [3,4] : \ u_k(x)>0\}$.  By the Borel--Cantelli Lemma, it remains to show that
$\sum_{k=1}^\infty \mu(E_k) < \infty$.

For every $\lambda \in \Lambda^1_k$ we have $\lambda \le \beta_{2^{k+1}}-2$ and hence
$\lambda \in 2^{-r_k} {\mathbb {N}}$.  Since $2^{r_k}$ divides $t_k$, it follows that
$$E_k {\subset}  \Big (\bigcup_{j=1}^\infty \Big [\frac{j}{t_k}-\frac1{\mathbf{n}t_k},\frac{j}{t_k}+\frac1{\mathbf{n}t_k}\Big ]\Big )\cap [3,4]:=G_k.$$
Since $\mu(G_k)=\frac2{\mathbf{n}}=\frac1{2^{k-1}}$, it follows that $\sum_{k=1}^\infty \mu(E_k)<\infty$.
\end{proof}
Hence the proofs of Claim \ref{34} and  of Theorem \ref{type1plusalg}
are complete.
\end{proof}

Looking at Theorems \ref{alg indep} and \ref{type1plusalg}, one might guess that
any discrete set $\Lambda$ containing infinitely many algebraically independent numbers is type 2.  As our next result (Theorem \ref{alg indep type 1})
 shows, this is not the case:

\begin{theorem}\label{alg indep type 1}
There exists a discrete set $\Lambda$ which is type 1 and which includes infinitely many algebraically independent numbers.
\end{theorem}

\begin{proof}  Let $\{\alpha_1,\alpha_2,\dots\}$ be a sequence of algebraically independent irrational numbers. For each $k,n \in  {\mathbb {N}}$ we define
$$ A_{n,k}=\Big \{\alpha_n+\frac{j}{2^k}: \ j\in  {\mathbb {Z}} \text{     and     } 0 < \alpha_n+\frac{j}{2^k} < 1\Big \}.$$

Let
 $P = \{(i,j)\in  {\mathbb {N}} \times  {\mathbb {N}} : \ i \le j \}$.
and define an anti-lexicographical ordering on $P$ as follows:
$$ (i,j) < (i',j') \text{     if either     } j < j' \text{     or     } (j=j' \text{     and     } i < i').$$

Now define $\{A_k\}_{k\in  {\mathbb {N}}}$ so that

$$ \text{     for each     } (i,j)\in P \text{     there exists     } k \in  {\mathbb {N}} \text{     such that     } A_k=A_{i,j},$$
and
$$ \text{     if     } k<k' \text{     and     } A_k=A_{i,j} \text{     and     }A_{k'}=A_{i',j'}, \text{     then     } (i,j)<(i',j').$$
For each $k \in  {\mathbb {N}}$ we also define $B_k=\cup_{n=1}^k A_n$.

We are now ready to define $\Lambda$.  For each $k \in  {\mathbb {N}}$ we define
$$\Lambda_k=\cup_{i=0}^{2^k-1}(B_k+2^k+i),$$
and let $\Lambda=\cup_{k=1}^\infty \Lambda_k$.

Note that 
\begin{equation}\label{(*24)}
\text{$(\Lambda\cap [n,n+1))+1\subset \Lambda \cap [n+1,n+2)$ for all $n \in  {\mathbb {N}}$}
\end{equation}
 and
 $$\text{  $(\Lambda\cap [n,n+1))+1= \Lambda \cap [n+1,n+2)$ if $n\neq 2^k.$  }$$

We continue with a few more definitions:


\begin{definition}\label{*defper}
 Given $i\in  {\mathbb {N}}$,
 and $n\in  {\ensuremath {\mathbb Z}}$
 we say that $\Gamma$ is {\em$ \frac1{2^i}$ periodic on $[n,n+1]$} if
$$\Big (\Gamma\cap \Big [n+\frac{j-1}{2^i},n+\frac{j}{2^i}\Big ]\Big )+\frac{1}{2^i}=\Gamma\cap\Big [n+\frac{j}{2^i},n+\frac{j+1}{2^i}\Big ]
\text{     for     }j=1,2,\dots,{2^i}-1.$$
\end{definition}

\begin{lemma}\label{periodicity}
Suppose that  $i\in  {\mathbb {N}}$.  If $n\ge N(i):=2^{\frac{(i-1)\cdot i}2+1}$, then $\Lambda$ is $\frac1{2^i}$ periodic on $[{n},{n+1}]$.
\end{lemma}

\begin{proof} The proof is straightforward and left to the reader. \end{proof}
\bigskip

In order to get a contradiction we now suppose that $\Lambda$ is type 2.  Then by Theorem \ref{char func} we can find a measurable set $S$ and a characteristic function
$f=\mathbf{1}_S$ such that $\mu(C(f,\Lambda))>0$ and $\mu(D(f,\Lambda))>0$.

Observe that for any $k\in  {\ensuremath {\mathbb N}}$ the set $( {\Lambda}+1/2^{k}) {\setminus}  {\Lambda}$ is a finite set and hence
condition $(*)$ of Proposition \ref{*proptrans} is satisfied. Hence $C(f, {\Lambda})$
is a right half-line and $D(f, {\Lambda})$ is a left half-line modulo sets of measure zero.

 Then we can choose intervals $I_C$ and $I_D$ of unit length such that
 \begin{equation}\label{*Dfl}
\text{  $\mu(D(f,\Lambda)\cap I_D ) = 1 \text{   and     } \mu(C(f,\Lambda)\cap I_C) = 1.$}
 \end{equation}
We assume without loss of generality that $I_C=[0,1]$ and
$I_D=[-N,-(N-1)]$ for some $N \in  {\mathbb {N}}$, where
\begin{equation}\label{N K0}
N \ge 3.
\end{equation}
Since $f(x)>0$ implies that $f(x)=1$, we can choose $\mathbf{C}\subset C(f,\Lambda)\cap I_C$ and $M\in  {\mathbb {N}}$ such that
 \begin{equation}\label{C ineq}
\mu(\mathbf{C})> 0.8
 \end{equation}
and
 \begin{equation}\label{misses C}
\text{     for all     }\lambda \in \Lambda \cap [M,\infty) \text{     and for all     } x\in \mathbf{C} \text{     we have     } \ x+\lambda \notin S.
 \end{equation}

We also assume that $M$ is chosen so that
 \begin{equation}\label{M N ineq}
M > 2N.
 \end{equation}
We define $E=I_C \backslash \mathbf{C}$ and $\mathbf{D}=E- N$ and for each $n \in  {\mathbb {N}}$
let
$$S_n^\prime=S \cap I_n \text{  where  } I_n= [n,n+1),$$
 $$\text{  $S_n=\{y \in I_n: \ (y-\Lambda) \cap \mathbf{C  } =\emptyset\}$,\qquad $\widetilde{S}_n=S_n-n$,}$$
$$\text{  $D_n=(S_n-\Lambda) \cap I_D$  } .$$
 We also define $D_n^\prime=D_n \backslash \mathbf{D}=D_{n} {\setminus} (E- N)$ and we let
 $D_n^{\prime\prime}=D_n \cap \mathbf{D}$ so $D_n = D_n^\prime \cup D_n^{\prime\prime}$.  Note that by our choice of $M$
and \eqref{(*24)} we have $S_n^\prime \subset S_n$ for all $n > M$.  For each $n \in  {\mathbb {N}}$ set
$$\Gamma_n=\Lambda\cap [n-1,n+1).$$
Observe that
\begin{equation}\label{*SnG}
S_{n}=\{y\in I_n: (y- {\Gamma}_{n})\cap \mathbf{C}= {\emptyset} \}\text{  and  }D_{n}=(S_{n}-\Gamma_{n+N})\cap I_{D}.
\end{equation}

For the remainder of the proof we assume that $n > M$. Observe that by \eqref{M N ineq} we have $n < n+N < \frac32 n$. Choose $m\in  {\mathbb {N}}$ such that $2^m < n \le 2^{m+1}$.    It follows that
 \begin{equation}\label{gamma near}
\Gamma_n,\Gamma_{n+N} \subset \Lambda_m \cup \Lambda_{m+1}.
\end{equation}
Let $p$ be the largest integer such that
$$\text{  $\Lambda$ is $\frac1{2^p  }$ periodic on $I_{n-1}$.}$$
Let $V = \{2^k : \, k \in  {\mathbb {N}}\}$.  We make the following useful observations:
 \begin{equation}\label{contain gamma}
\Gamma_n+ 1\subset \Gamma_{n+1},
 \end{equation}
 \begin{equation}\label{Snsubset}
 S_{n}\subset S_{n-1}+1 \text{  and hence  }\widetilde{S}_{n} {\subset}
 \widetilde{S}_{n-1},
 \end{equation}
 \begin{equation}\label{gamma equal} \Gamma_n+1= \Gamma_{n+1} \text{  as long as  } \{n,n+1\}\cap V = \emptyset,
 \end{equation}
\begin{equation}\label{gamma periodic}
\Gamma_n \text{  is  } \frac1{2^p}\text{  periodic on  } [n-1,n+1] \text{  as long as  }n \notin V,
\end{equation}
\begin{equation}\label{Sn periodic}
S_n \text{  is  } \frac1{2^p} \text{  periodic on  } I_n \text{  as long as  } n \notin V,
\end{equation}
\begin{equation}\label{Sn equal}S_n+1 = S_{n+1} \text{  as long as  } \{n,n+1\} \cap V =\emptyset.
\end{equation}

Let $\widetilde{\Gamma}_n=\Gamma_{n+N}\backslash (\Gamma_n+N)$ and note that by the definition of $ {\Lambda}$ and $p$ and \eqref{gamma near} we have
\begin{equation}\label{less than 2}
\#\Big (\widetilde{\Gamma}_n \cap\Big [n+N-1+\frac{j-1}{2^p},n+N-1+\frac{j}{2^p}\Big ]\Big )\le 2
 \end{equation}
$\text{  for  }j=1,2,\dots,2^{p+1}.$
Define
$$T_n=(\Lambda+D_n^\prime)\cap S_n.$$

Next we prove that
$T_n= (\widetilde{\Gamma}_n+D_n^\prime)\cap S_n$.
From $x\in D_{n}'$ it follows that $x\not \in \mathbf{D}=(I_{C} {\setminus} \mathbf{C})-
 N$ and hence $x\in \mathbf{C}-N$.
 This implies that $x+N+\lambda\not \in S_{n}$ for $\lambda\in \Gamma_{n}$.
 On the other hand, obviously $(\Lambda+D_n^\prime)\cap S_n = (\Gamma_{n+N}+D_n^\prime)\cap S_n \supset (\widetilde{\Gamma}_n+D_n^\prime)\cap S_n.$

Now we show that
\begin{equation}\label{*tnwtgn}
T_{n}-\widetilde{{\Gamma}}_{n}\supset D_{n}'.
\end{equation}
Indeed, if $x\in D_{n}'$ then there exists
$\lambda\in \widetilde{{\Gamma}}_{n}$ such that $x+\lambda=y\in S_{n}.$
Then $y\in (\widetilde{\Gamma}_n+D_n^{\prime})\cap S_n=T_{n}$
and $x=y-\lambda\in T_{n}-\widetilde{{\Gamma}}_{n}.$

We claim that
\begin{equation}\label{TnSn}
(T_n+N) \cap S_{n+N}=\emptyset.
\end{equation}

To prove this claim let $y'\in T_n + N$.  Then $y' = y+N$, where $y \in T_n$.  Thus, we can choose
$x \in D_n^\prime$ and $\lambda \in \widetilde{\Gamma}_n$ such that $x+\lambda=y$.  Since $x+ N \in \mathbf{C}$ and $y'=x+N+\lambda$, we see that $y' \notin S_{n+N}$, as desired. 

Let $\widetilde{T}_n=T_n-n \subset \widetilde{S}_n.$
Observe that since 
by \eqref{Snsubset},
$\widetilde{S}_{n+1}\subset \widetilde{S}_n$, from \eqref{TnSn} and the fact that $\widetilde{T}_n \subset \widetilde{S}_n$, we conclude that
\begin{equation}\label{no intersect}
\widetilde{T}_{n+kN} \cap \widetilde{T}_n 
 {\subset}
\widetilde{S}_{n+kN} \cap \widetilde{T}_n
 {\subset}
\widetilde{S}_{n+N} \cap \widetilde{T}_n  
= \emptyset \text{  for all  }k \in  {\mathbb {N}}.
\end{equation}

We next examine several cases depending on the membership of $n$ and $n+N$ in $V$ as equations (\ref{gamma equal}-\ref{Sn equal}) show
that these are the exceptional cases.

First suppose that $n \in V $.  In this case, from \eqref{N K0}, \eqref{M N ineq} and $n>M$ we conclude that
$\{n+N-1,n+N\} \cap V = \emptyset $ and therefore
by \eqref{gamma equal},
 $\Gamma_{n+N}=\Gamma_{n+N-1}+1$.  Since we also have $S_{n}\subset S_{n-1}+1$, it follows that
\begin{equation}\label{*dncc1}
 D_n =(S_n-\Gamma_{n+N})\cap I_D \subset (S_{n-1}-\Gamma_{n+N-1})\cap I_D =D_{n-1}
 \end{equation}
 and hence
 \begin{equation}\label{*dncc2}
D_n^\prime=D_{n} {\setminus} \mathbf{D} \subset D_{n-1} {\setminus} \mathbf{D}=D_{n-1}^\prime.
 \end{equation}

Now suppose that $n+N \in V$.  In this case, using
 \eqref{N K0}, \eqref{M N ineq} and $n>M$ 
 again, we see that $\{n,n+1\} \cap V =\emptyset $ and therefore  \eqref{Sn equal}  holds
and hence $S_{n}+1=S_{n+1}$.
By \eqref{contain gamma} we also have
  $\Gamma_{n+N}+1\subset \Gamma_{n+N+1}$.
  Therefore, we obtain
\begin{equation}\label{*dncd1}
 D_n =(S_n-\Gamma_{n+N})\cap I_D \subset (S_{n+1}-\Gamma_{n+N+1})\cap I_D =D_{n+1}
 \end{equation}
 and hence
 \begin{equation}\label{*dncd2}
D_n^\prime=D_{n} {\setminus} \mathbf{D} \subset D_{n+1} {\setminus} \mathbf{D}=D_{n+1}^\prime.
 \end{equation}

Finally, suppose that $\{n,n+N\} \cap V  = \emptyset.$  In this case we have the following:
$$ S_n \text{  is  } \frac1{2^p} \text{  periodic on  }I_n,$$
$$\Gamma_{n+N} \text{  is  } \frac1{2^p} \text{  periodic on  }  [n+N-1,n+N+1 ] $$ and
$$\Gamma_{n} \text{  is  } \frac1{2^p} \text{  periodic on  }  [n-1,n+1 ].$$
  It follows that $\widetilde{\Gamma}_n$ is $\frac1{2^p}$ periodic on $[n+N-1,n+N+1]$ and thus $T_n$ is $\frac1{2^p}$ periodic on $I_n$.  
  Therefore $T_n-\widetilde{\Gamma}_n$ is $\frac1{2^p}$ periodic on $I_D$.  Using this fact, along with \eqref{less than 2} and \eqref{*tnwtgn} we conclude that
\begin{equation}\label{mu of D}
\mu(D_n^\prime)\le 2 \mu(T_n)=2 \mu(\widetilde{T}_n).
\end{equation}

Now let  $R=\{j > M: \, \{j,j+N\} \cap V =\emptyset \}$ and for each $k=1,2,\dots,N$ define
$N_k=\{M+k,M+k+N,M+k+2N,\dots\}$.  Then using  \eqref{no intersect} 
and \eqref{mu of D} 
we deduce that for
$k=1,2,\dots ,N$ we have
\begin{align}\label{list eqn}\sum_{j \in N_k \cap R} \mu(D_j^\prime)&\le 2 \sum_{j=0}^\infty \mu(\widetilde{T}_{M+k+jN})\\\ \nonumber
       &=2\mu(\cup_{j=0}^\infty \widetilde{T}_{M+k+jN})\\ 
       &\le 2. \label{list eqn end}
\end{align}
Now let  $R_1=\{j >M : \, j \in V\}$ and
 $R_2=\{j >M : \, j+N \in V\}$.  Then $R_1-1 \subset R$ and $R_2+1 \subset R$.
Using this fact along with \eqref{*dncc2}, \eqref{*dncd2} and (\ref{list eqn}-\ref{list eqn end}) we see that for $k=1,2,\dots N$ we have
$\sum_{j\in N_k \cap R_i} \mu(D_j^\prime) \le 2 $ for $i=1,2$.  Putting this together with
(\ref{list eqn}-\ref{list eqn end}) we deduce that
$$\sum_{j=M+1}^\infty \mu(D_j^\prime) < 6N < \infty .$$

Let $G=\{x \in I_D \backslash \mathbf{D} : \ \sum_{\lambda \in \Lambda} f(x+\lambda)=\infty \}$.  Then the above inequality and the Borel--Cantelli Lemma tell us that $\mu(G)=0$.  Therefore, we have shown that $\mu(D(f,\Lambda)\backslash \mathbf{D})=0$ and hence it follows that
$\mu(D(f,\Lambda )\cap I_D ) \le \mu(\mathbf{D}) < 0.2 $ which  contradicts
\eqref{*Dfl}, as desired.  \end{proof}

\section{Unions and Minkowski sums}\label{secminko}

\begin{proposition}\label{*proptos}
If $\Lambda_1,\Lambda_2\subset {\ensuremath {\mathbb R}}$ are type 1 sets then $\Lambda_1\cup\Lambda_2$ is also type 1.
\end{proposition}
\begin{proof}
Let $f\colon {\ensuremath {\mathbb R}}\rightarrow {\ensuremath {\mathbb R}}$ be a non-negative measurable function. For every $x\in {\ensuremath {\mathbb R}}$ we have
\[
\max\left(\sum_{\lambda\in\Lambda_1} f(x+\lambda), \sum_{\lambda\in\Lambda_2} f(x+\lambda)\right) \le \sum_{\lambda\in\Lambda_1\cup\Lambda_2} f(x+\lambda) \le \sum_{\lambda\in\Lambda_1} f(x+\lambda) + \sum_{\lambda\in\Lambda_2} f(x+\lambda),
\]
hence
\[
D(f,\Lambda_1),D(f,\Lambda_2)\subset D(f,\Lambda_1\cup\Lambda_2) \subset D(f,\Lambda_1)\cup D(f,\Lambda_2),
\]
i.e. if $\mu( {\ensuremath {\mathbb R}}\setminus D(f,\Lambda_1))=0$ or $\mu( {\ensuremath {\mathbb R}}\setminus D(f,\Lambda_2))=0$ then $\mu( {\ensuremath {\mathbb R}}\setminus D(f,\Lambda_1\cup\Lambda_2))=0$, and $\mu(D(f,\Lambda_1))=\mu(D(f,\Lambda_2))=0$ implies $\mu(D(f,\Lambda_1\cup\Lambda_1))=0$, thus $\Lambda_1\cup\Lambda_2$ is type 1.
\end{proof}

\begin{proposition}\label{2-esek unioja}
There exist two decreasing gap asymptotically dense type 2 sets $\Lambda_1$ and $\Lambda_2$ such that there union is type 1.
\end{proposition}
\begin{proof}
Set
$$
\Lambda_1 = \bigcup_{i=0}^\infty \left(\left(\bigcup_{n=2^{2i}}^{2^{2i+1}-1} (2^{-n}\cdot {\ensuremath {\mathbb Z}}) \cap [n,n+1)\right) \cup \left([2^{2i+1},2^{2i+2})\cap (2^{-2^{(2i+1)}}\cdot {\ensuremath {\mathbb Z}})\right) \right),
$$
and
$$
\Lambda_2 = \bigcup_{i=0}^\infty \left( \left(\bigcup_{n=2^{2i+1}}^{2^{2i+2}-1} \left((2^{-n}\cdot {\ensuremath {\mathbb Z}}) \cap [n,n+1)\right)\right) \cup
\left([2^{2i},2^{2i+1})\cap (2^{2^{-2i}}\cdot {\ensuremath {\mathbb Z}})\right) 
\right).
$$
For every $i\in {\ensuremath {\mathbb N}}$ we have
$$
\frac{\#\left(\Lambda_1\cap[2^{2i},2^{2i}+1)\right)}{\#\left(\Lambda_1\cap[2^{2i}-1,2^{2i})\right)} = \frac{2^{2^{2i}}}{2^{2^{2i-1}}} = 2^{2^{2i-1}},
$$
and
$$
\frac{\#\left(\Lambda_2\cap[2^{2i+1},2^{2i+1}+1)\right)}{\#\left(\Lambda_2\cap[2^{2i+1}-1,2^{2i+1})\right)} = \frac{2^{2^{2i+1}}}{2^{2^{2i}}} = 2^{2^{2i}}
$$
hence $\Lambda_1$ and $\Lambda_2$ are type 2 by Theorem \ref{ez is 2-es}.

From the definition of these sets
$$
\Lambda_1\cup\Lambda_2 = \bigcup_{n=0}^\infty [n,n+1)\cap 2^{-n}\cdot {\ensuremath {\mathbb Z}},
$$
which is a type 1 set according to Theorem \ref{mkthm}.
\end{proof}

\begin{theorem}\label{*minksum}
If the sets $\Lambda=\{\lambda_0,\lambda_1,\ldots\}$ and $\Lambda'=\{\lambda'_0,\lambda'_1,\ldots\}$ are type 1 then the Minkowski sum $\Lambda+\Lambda'$ is
also type 1.
\end{theorem}
\begin{proof}
We can assume that $\lambda_0=\lambda'_0=0$ as a translation does not change the type of a set.

Take a measurable characteristic function $f\colon  {\ensuremath {\mathbb R}}\rightarrow {\ensuremath {\mathbb R}}$ (by Theorem \ref{char func} it is enough to study characteristic functions). If $\sum_{\lambda\in\Lambda} f(x+\lambda)$ diverges for almost every $x\in {\ensuremath {\mathbb R}}$ then $\sum_{\widetilde{\lambda}\in\Lambda+\Lambda'} f(x+\widetilde{\lambda})$ also diverges for almost every $x\in {\ensuremath {\mathbb R}}$, since $\Lambda+\Lambda'$ contains $\Lambda$.

If $\sum_{\lambda\in\Lambda} f(x+\lambda)$ converges almost everywhere, then the function $g$ 
 defined by $g(x):=\sum_{\lambda\in\Lambda} f(x+\lambda)$  is a non-negative extended real valued function (that is $g\colon  {\ensuremath {\mathbb R}}\rightarrow[0,\infty]$), and it has a finite value almost everywhere. For every $x\in {\ensuremath {\mathbb R}}$
\begin{equation}\label{fg1}
\sum_{\lambda'\in\Lambda'} g(x+\lambda') = \sum_{\widetilde{\lambda}\in\Lambda+\Lambda'} \#\{\lambda\in\Lambda : \widetilde{\lambda}-\lambda\in\Lambda'\} \cdot f(x+\widetilde{\lambda}),
\end{equation}
thus using that $f$ is a characteristic function we obtain
\begin{equation}\label{fg1.2}
\sum_{\lambda'\in\Lambda'} g(x+\lambda')<\infty \text{  if and only if  } \sum_{\widetilde{\lambda}\in\Lambda+\Lambda'} f(x+\widetilde{\lambda})<\infty.
\end{equation}

For every $x\in {\ensuremath {\mathbb R}}$ let
$$
g^*(x) :=
\begin{cases}
g(x) & g(x)<\infty \\
0 & g(x)=\infty.
\end{cases}
$$
As $g^*$ and $g$ agree almost everywhere, by \eqref{fg1.2} for almost every $x\in {\ensuremath {\mathbb R}}$ we have
\begin{equation}\label{fg2}
\sum_{\lambda'\in\Lambda'} g^*(x+\lambda')<\infty \text{  if and only if  } \sum_{\widetilde{\lambda}\in\Lambda+\Lambda'} f(x+\widetilde{\lambda})<\infty.
\end{equation}
Since $\Lambda'$ is type 1, $\sum_{\lambda'\in\Lambda'} g^*(x+\lambda')$ converges almost everywhere or diverges almost everywhere, hence $\sum_{\lambda\in\Lambda+\Lambda'} f(x+\lambda)$ also converges almost everywhere or diverges almost everywhere according to \eqref{fg2}.
\end{proof}

\begin{proposition}\label{*mink2}
There is a type 2 set $\Lambda$ such that for every $\Lambda'=\{\lambda'_0,\lambda'_1,\ldots\}$ the Minkowski sum $\Lambda+\Lambda'$ is type 2.
\end{proposition}
\begin{proof}
Let $\Lambda$ be a type 2 set which is not contained by a type 1 set as guaranteed by Theorem \ref{always type 2}. We can assume that $\lambda_0=\lambda'_0=0$ as a translation does not change the type of a set. Then $\Lambda\subset\Lambda'+\Lambda$, hence $\Lambda'+\Lambda$ is also type 2.
\end{proof}

It is useful to point out that it is easy to construct examples of type 2 sets with type 1 sum or a type 2 and a type 1 set with type 1 sum.  For instance, we take $\Lambda_1$ and $\Lambda_2$ from the proof of Proposition \ref{2-esek unioja} and let $\Lambda=\Lambda_1\cup\Lambda_2$. We know that $\Lambda_1$ and $\Lambda_2$ are type 2 and $\Lambda$ is type 1. All of them contain $0$ hence $\Lambda$ is a subset of $\Lambda_1+\Lambda_2$ and $\Lambda_1+\Lambda$.
 Since $\lambda+\Lambda\subset\Lambda$ for every $\lambda\in\Lambda$, we also have $\Lambda_1+\Lambda_2, \Lambda_1+\Lambda \subset \Lambda$,
 therefore $ {\Lambda}= {\Lambda}_{1}+ {\Lambda}_{2}= {\Lambda}_{1}+ {\Lambda}$ is a type $1$ set.

\section{Acknowledgements}

Z. Buczolich thanks the R\'enyi Institute where he was
a visiting researcher for the academic year 2017-18.

B. Hanson would like to thank the Fulbright Commission, the Budapest Semesters in Mathematics, and the R\'enyi Institute for their generous support during the Spring of 2018, while he was visiting Budapest as a Fulbright scholar.


\end{document}